\documentclass[12pt]{article}
\usepackage{amsthm}

\usepackage[english]{babel}

\usepackage{etoolbox}
\AtBeginEnvironment{thebibliography}{%

  \setlength{\parskip}{0pt}
  \setlength{\baselineskip}{10pt}
    \setlength{\itemsep}{0pt}
}

\usepackage[letterpaper,top=2cm,bottom=2cm,left=3cm,right=3cm,marginparwidth=1.75cm]{geometry}

\usepackage{xcolor}
\usepackage{amsmath}
\usepackage{tikz}
\usetikzlibrary{arrows.meta, positioning, calc}

\newtheorem{theorem}{Theorem}[section]

\newtheorem{lemma}{Lemma}

\newtheorem{definition}{Definition}

\newtheorem{problem}{Problem}
\newtheorem{conjecture}{Conjecture}
\newtheorem{claim}{Claim}

\usepackage{amsmath}
\usepackage{graphicx}
\usepackage[colorlinks=true, allcolors=blue]{hyperref}
\usepackage{tikz}
\usetikzlibrary{arrows.meta, shapes.geometric}
\usetikzlibrary{positioning, arrows, shapes, calc}
\usepackage{graphicx}
\usepackage{subcaption}
\usepackage{float}
\usepackage{amssymb}
\usepackage{amsmath}
\usepackage{enumitem}
\usepackage{mathtools}

\usepackage{pifont}

\usepackage{bm}
\usepackage{extarrows}
\usepackage{mathrsfs}

\usepackage{lipsum}

\title{A Chv\'atal--Erd\H{o}s type condition for supereulerian digraphs  with $\alpha_{2}(D)=4$}
\author{Zirui Liu$^{1}$,   Jia Zhou$^{2}$\footnote{Corresponding author. E-mail address: jiazhou\_99@163.com.}\\[2mm]
\small $^{1}$School of Mathematics, Shandong University, Jinan 250100, China\\
\small $^{2}$School of Mathematics and Statistics, Ningxia University, Yinchuan 750021, China}
\date{}

\begin{document}
	\maketitle

	\begin{abstract}
        A digraph is \textbf{supereulerian} if it contains a spanning closed trail. Let $\alpha_2(D)$ denote the maximum cardinality of a vertex set inducing no 2-cycle. In this paper, we study the supereulerianity of digraphs with arc-strong connectivity at least 2. We identify a family $\mathcal H$ of nonsupereulerian 2-arc-strong digraphs and prove that every 2-arc-strong digraph $D$ with $\alpha_2(D)\leq 4$ is supereulerian if and only if $D\notin\mathcal H$. In particular, every digraph $D$ satisfying $\alpha_2(D) \le 4$ and $\lambda(D) \ge 3$ is supereulerian.
	\end{abstract}
\vspace{1ex}
\vspace{1ex}

	
	\section{Introduction}
	Supereulerianity, a natural relaxation of Hamiltonicity, represents a fundamental and widely studied property in graph theory. A \textbf{Hamilton cycle} refers to a cycle containing all vertices. Let $\boldsymbol{\kappa(G)}$ and $\boldsymbol{\alpha(G)}$ denote the vertex-connectivity and the independence number of a graph $G$, respectively. The classical theorem of Chv\'atal and Erd\H{o}s states that every graph $G$ of order at least three with $\kappa(G)\geq \alpha(G)$ contains a Hamilton cycle \cite{ChvatalErdos1972}. This result established a fundamental connection between connectivity and independence and initiated a broad line of work on Chv\'atal--Erd\H{o}s type conditions; see, for example, \cite{aishi2026, jackson,JacksonOrdaz1990,KuhnOsthus2012}.

    The directed analogue of the Chvátal--Erd\H{o}s theorem is substantially more delicate; even the following basic question remains open. Here, { $\alpha_0(D):=\max\{|S|: S\subseteq V(D)$, $(x,y)\notin A(D)\text{ for all distinct }x,y\in S\}$} is defined as the \textbf{independence number} of $D$, which constitutes the most natural generalization of the independence number $\alpha(G)$ for undirected graphs.

    \begin{problem}[K\"uhn and Osthus \cite{KuhnOsthus2012}]
       Is there a function $f_0(k)$ such that, for every positive integer $k$, every digraph $D$ with $\kappa(D)\geq f_0(k)$ and $\alpha_0(D)\leq k$ contains a Hamilton cycle?
    \end{problem}

In contrast, the situation is different for the alternative $\boldsymbol{\alpha_2}$\textbf{-independence number} of a digraph $D$ which is introduced by Thomassen \cite{Thom81}, defined by
\[
\alpha_2(D)=\max\{|S|:S\subseteq V(D),\ D[S]\text{ contains no directed $2$-cycle}\}.
\] In 1987, Jackson \cite{jackson} proved that $f_2(a)\leq 2^a(a+2)!$, where $f_2(a)$ is the least integer $K$ such that every $K$-connected digraph $D$ with $\alpha_2(D)\leq a$ is Hamiltonian. Ai and Shi \cite{aishi2026} improved this factorial bound to $f_2(a)\leq 12000a^2$.  Subsequently, Li and Ning \cite{LiNing2026} proved
$f_2(a) = O\left( \frac{a(\log a)^4}{(\log \log a)^2} \right)$.

 A natural question is to consider the Chv\'atal--Erd\H{o}s type condition for supereulerianity. In 2011, Bang-Jensen and Thomass\'e formulated the following conjecture (unpublished; see \cite{BangJensenMaddaloni2015}).

	\begin{conjecture}[Bang-Jensen and Thomass\'e; see \cite{BangJensenMaddaloni2015}]
		Let $D$ be a digraph. If $\lambda(D)\geq\alpha_0(D)$, then $D$ is supereulerian.\label{conjecture1}
	\end{conjecture}

Notably, it remains open whether there exists some constant $K$ such that $\lambda(D)>K$ and $\alpha_0(D)=2$ implies that $D$ is supereulerian. Bang-Jensen and Maddaloni established several auxiliary results.
In particular, they proved that every digraph $D$ satisfying
$\lambda(D)\geq\alpha_0(D)$ has an Eulerian factor,
that is, a spanning subdigraph whose every component is supereulerian
\cite{BangJensenMaddaloni2015}. Related sufficient conditions in terms of matching number and bipartite structure were obtained in \cite{AlgefariLai2016,ZhangLiuWangLai2018,WeiLai2025}.

 Towards Conjecture~\ref{conjecture1} with $\alpha_2(D)$ in place of $\alpha_0(D)$, {Zhang, Yang, Lai and Liu \cite{ZhangYangLaiLiu2024}} established a necessary and sufficient condition for a strong digraph with $\alpha_2(D)=2$ to be {strongly supereulerian-connected}. A digraph $D$ is \textbf{strongly supereulerian-connected} if, for each pair $x,y\in V(D)$, $D$ admits a spanning $(x,y)$-trail. Most recently, {Liu , Zhao, Yang and Lai \cite{LiuZhaoYangLai2027}}   obtained a characterization of strong digraphs with $\alpha_2(D)=3$,
which yields the following result.
  \begin{theorem}[\cite{LiuZhaoYangLai2027}]
		{ Let $D$ be a strong digraph with $\alpha_2(D)=3$, and let $\mathcal{H}_1$ denote the exceptional family arising from the characterization in \cite{LiuZhaoYangLai2027}. Then $D$ is supereulerian if and only if $D\notin\mathcal{H}_1$.}
	\end{theorem}

	In this paper, we characterize the supereulerianity
of strong digraphs $D$ with $\alpha_2(D)=4$. The exceptional family $\mathcal H$ is defined in Section 2.

	\begin{theorem}\label{main}
		Let $D$ be a strong digraph with $\alpha_2(D)=4$. Then the following statements hold:
		\begin{itemize}
		\item[(i)]  If $\lambda(D)\geq2$, then $D$ is supereulerian if and only if $D \notin \mathcal{H}$.
		\item[(ii)] If $\lambda(D)\geq3$, then $D$ is supereulerian.
		\end{itemize}
		\label{theorem1.2}
	\end{theorem}

   \noindent \textbf{Organization.}  	
	The  paper is organized as follows. Section 2  introduces the notation and preliminary lemmas used throughout the paper, and defines the exceptional
family $\mathcal H$ together with its basic properties. Section 3 establishes the key reduction lemma  to be used in the proof of Theorem~\ref{theorem1.2}. Section 4 proves Theorem~\ref{theorem1.2} by analyzing the connected components of the symmetric core. We conclude in Section 5 with several open problems.	

	\section{Preliminaries}

Undefined notation follows \cite{book}. All digraphs are
finite and loopless, and no multiple arcs are allowed. When parallel arcs are allowed, we call the resulting
objects \textbf{directed multigraphs}. For a digraph $D$, we
write $\boldsymbol{V(D)}$ and $\boldsymbol{A(D)}$ for its vertex set and arc set, respectively, and
write $\boldsymbol{|D|}:=|V(D)|$ for the  \textbf{order} of $D$. If
$U\subseteq V(D)$, then $\boldsymbol{D[U]}$ denotes the subdigraph of $D$ induced by
$U$, and $\boldsymbol{D-U}:=D[V(D)\setminus U]$. For a vertex $v\in V(D)$, let
\[
 \boldsymbol{N_D^+(v)}:=\{x\in V(D):vx\in A(D)\},
 \qquad
 \boldsymbol{N_D^-(v)}:=\{x\in V(D):xv\in A(D)\}
\]
denote the \textbf{out-neighbourhood} and \textbf{in-neighbourhood} of $v$, respectively, and
let $ \boldsymbol{d_D^+(v)}:=|N_D^+(v)|$ and $
 \boldsymbol{d_D^-(v)}:=|N_D^-(v)|.$
For vertex subsets $X,Y\subseteq V(D)$, define
		\begin{equation*}
			\boldsymbol{(X,Y)_D}=\{(x,y)\in A(D) : x\in X,y\in Y\}.
		\end{equation*}
If $X=\{v\}$, we write
$(v,Y)_D=\{(v,y)\in A(D):y\in Y\}$.
Similarly, if $Y=\{w\}$, we write
$(X,w)_D=\{(x,w)\in A(D):x\in X\}$.

A digraph is \textbf{strong} if, for every ordered pair
of distinct vertices $u,v$, there is a $(u,v)$-path. The \textbf{arc-strong connectivity} $\lambda(D)$ of
a digraph $D$ is the largest integer $l$ with
$0\leq l\leq |V(D)|-1$ such that $D-F$ is strong
whenever $F\subseteq A(D)$ and $|F|<l$.
Thus $D$ is $k$-arc-strong if $\lambda(D)\geq k$. Its \textbf{strong connectivity} $\kappa(D)$ is the largest integer $k$ with $0\leq k\leq |V(D)|-1$ for which $D-S$ is strong whenever $S\subseteq V(D)$ and $|S|<k$. We set $\lambda(D)=\kappa(D)=0$ for $|V(D)|\leq1$ or nonstrong digraph $D$. A \textbf{complete digraph} contains a 2-cycle between each pair of distinct vertices. Hence a complete digraph of order $n\geq2$ has $\lambda(D)=\kappa(D)=n-1$.

Let $\boldsymbol{\mathbb N}$  denote the set of positive integers. For integers $a\leq b$, let $[a,b]:=\{i\in\mathbb Z:a\leq i\leq b\}$, $[b]:=[1,b]$. Given sets $A$ and $B$, the symbol $\boldsymbol{A\mathbin{\dot\cup} B}$ stands for their disjoint union, i.e., $A\cap B=\emptyset$.

 %If $F$ is a subdigraph of $D$, we also write $\boldsymbol{D-F}:=D-V(F)$.

A digraph is \textbf{symmetric} if $xy\in A(D)$ implies $yx\in A(D)$. The symmetric core was introduced as a tool for studying spanning trails in \cite{Liu2021} and was further investigated in \cite{LiuYangLaiZhang2021}. The \textbf{symmetric core} $J(D)$ of $D$ is the
spanning symmetric subdigraph whose arc set is  $A(J(D))=\{xy\in A(D):yx\in A(D)\}$. Let $\boldsymbol{G_D}$ denote the underlying graph of $J(D)$. By definition, two vertices are adjacent in $G_D$ precisely when they are in a 2-cycle in $D$. Consequently,
\begin{equation*}
  \alpha(G_D)=\alpha_2(D) \text{  and }\kappa(G_D)\leq\kappa(D).
\end{equation*}
A connected symmetric digraph is strong. Moreover, a spanning tree of its underlying graph, with both orientations of every tree edge, is a connected Eulerian spanning subdigraph. Thus every connected symmetric digraph is supereulerian.

\subsection{Lemmas}

   { Liu, Yang, Lai and Zhang \cite{LiuYangLaiZhang2021}} established the following result, which will be used repeatedly.

	\begin{lemma}
		\cite[Lemma~3.2]{LiuYangLaiZhang2021}  Let $s\geq 1$ be an integer, let $D$ be a strong symmetric digraph with $\lambda(D)\geq s$, and let $x_1,x_2,\ldots,x_s$ and $y_1,y_2,\ldots,y_s$ be two vertex sequences of $D$, in which repetitions are allowed. Then $D$ has a connected spanning subdigraph $T'_D$ that is the arc-disjoint union of trails $T_1,T_2,\ldots,T_s$, where $T_i$ is an $(x_i,y_i)$-trail for every $i\in[s]$.
    \label{lemma1}
	\end{lemma}

    Lemma \ref{lemma1}  immediately implies that each strong symmetric digraph
	is supereulerian and strongly supereulerian-connected, with the singleton case supplied by the trivial trail. The following immediate consequence of Lemma~\ref{lemma1} will be used repeatedly.

	\begin{lemma}
		Let $s\in \mathbb N$. For $i\in[4]$, let $D_i$ be a complete subdigraph of $D$, and let $x^i_1,\ldots,x^i_s$ and $y^i_1,\ldots,y^i_s$ be two vertex sequences of $V(D_i)$, in which repetitions are allowed. If $|V(D_i)|\geq s+1$ or $|V(D_i)|=1$, then there exist pairwise arc-disjoint $(x^i_j,y^i_j)$-trails $T^i_{(x^i_j,y^i_j)}$, $j\in[s]$, (not required to be spanning) such that
			\[
			\bigcup_{j=1}^s V\bigl(T^i_{(x^i_j,y^i_j)}\bigr)=V(D_i).
			\]
		\label{lemma4}
	\end{lemma}

	\begin{proof}
	If $|V(D_i)|\ge s+1$, then $\lambda(D_i)=|V(D_i)|-1\ge s$, because $D_i$ is complete. Therefore the assertion follows from Lemma~\ref{lemma1}. If $|V(D_i)|=1$, say $V(D_i)=\{v\}$, then  $x^i_1=\cdots =x^i_s =y^i_1=\cdots =y^i_s=v$, and the $s$ trivial trails $\{v,\ldots, v\}$ have the required properties.
    \end{proof}

Lemmas \ref{lemma5} and \ref{lemma6} address the existence of arc‑disjoint trails in complete digraphs on 2 and 3 vertices, which are not covered by Lemma \ref{lemma4}.
\begin{lemma}\label{lemma5}
Let $s\in \mathbb N$. Let $K$ be a complete digraph on $\{a,b\}$, and let $(x_j,y_j)$, $j\in[s]$, be pairs of vertices of $K$ \emph{(}$x_j=y_j$ is possible\emph{)}. There exist pairwise arc-disjoint $(x_j,y_j)$-trails $T_j$, $j\in[s]$, with $V(K)= \bigcup_{j\in [s]} V(T_j)$ if and only if at most one pair is $(a,b)$ and at most one pair is $(b,a)$.
\end{lemma}

\begin{proof}

Every $(a,b)$-trail uses the unique arc $(a,b)$, and every $(b,a)$-trail uses the unique arc $(b,a)$, which proves necessity. Conversely, for each pair $(x_j,y_j)$, if $ x_j\neq y_j$, we let $ T_j:=(x_j,y_j)$, otherwise, let $ T_j:=x_j (=y_j)$. These trails are pairwise arc-disjoint as at most one pair is $(a,b)$ and at most one pair is $(b,a)$. If there is a pair $(x_j,y_j)$ with  $ x_j\neq y_j$, then $V(T_j)=V(K)$. Thus, assume that all pairs $(x_j,y_j)$ have $ x_j= y_j$. Furthermore, if there exist two pairs $ (x_j,y_j)$ and $(x_i,y_i)$ such that $ x_j=y_j=a$  and $x_i=y_i=b$, then $V(K)= V(T_i)\cup V(T_j)$. Thus, assume that $x_j=y_j=a$ for every $j$
(the case $x_j=y_j=b$ is similar).
Then $s-1$ trivial trails and one trail $aba$
give the desired $s$ trails.  This completes the proof.
\end{proof}

\begin{lemma}\label{lemma6}
Let $K$ be a complete digraph on three vertices, and let $(x_j,y_j)$, $j\in[3]$, be three pairs of vertices of $K$ (it is possible that $x_j=y_j$). There exist three pairwise arc-disjoint trails $T_1,T_2,T_3$ such that each $T_j$ is an  $(x_j,y_j)$-trail , $j\in[3]$, and $V(K)= V(T_1)\cup V(T_2)\cup V(T_3)$ if and only if there is no nonempty proper set $A\subset V(K)$ such that
\[
x_j\in A\ \text{and}\  y_j\in V(K)\setminus A,\ \text{ for each }j\in[3]\ \  (i.e., \ |(A,V(K)\setminus A)_K|
=2).
\]
\end{lemma}

\begin{proof}
We first prove the necessity. If such a set $A$ exists, then each of the three trails must use an arc from $A$ to $V(K)\setminus A$. As $K$ contains precisely two such arcs, a contradiction follows.

To establish the sufficiency, suppose that
$x_i=y_i$ for some $i\in[3]$.
By Lemma~\ref{lemma4} with $s:=2$, there exist
two pairwise arc-disjoint trails $T_1,T_2$ such that each $T_j$ is an  $(x_j,y_j)$-trail, $j\in[3]\setminus {i}$, and $V(K)= V(T_1)\cup V(T_2)$.
The trivial trail at $x_i$ completes the required
collection of three trails.

Hence, suppose  $x_j\neq y_j$ for each $j\in [3]$.
Since no such separating set $A$ exists,
there must exist $i,j\in[3]$ with $i\neq j$
such that $x_i=y_j$.  Without loss of generality, assume $x_1 = y_3$. Hence, choose the arc $(x_2,y_2)$ as the $(x_2,y_2)$-trail $T_{(x_2,y_2)}$.
Let $H$ be a digraph such that $V(H)=V(K)$ and $A(H)=A(K)\setminus \{(x_2,y_2)\}$. The symmetric core $J(H)$ is a strong symmetric
digraph with $\lambda(J(H))=1$.
By Lemma~\ref{lemma1}, $J(H)$ contains a spanning
$(x_3,y_1)$-trail $T_{(x_3,y_1)}$.
Since $x_1=y_3$ and $T_{(x_3,y_1)}$ is spanning, we may split it at $x_1$ into an $(x_3,y_3)$-trail and an $(x_1,y_1)$-trail. This means that three arc-disjoint trails have been found.
\end{proof}

Lemma~\ref{lemma7} gives a sufficient condition
for a digraph $D$ to be supereulerian in terms
of a vertex partition and a spanning closed trail
in the corresponding contracted digraph.

\begin{lemma}\label{lemma7}
Suppose $D$ is a digraph with a partition $V(D)=V(D_1)\dot\cup\cdots\dot\cup V(D_r)$, where each $D_i$ is a digraph. Let
$W=i_1i_2\cdots i_mi_1$
be a spanning closed trail on $[r]$. Suppose that, for each $t\in[m]$, an arc
\[
e_t=(y^{i_t}_t,x^{i_{t+1}}_{t+1})
\in (V(D_{i_t}),V(D_{i_{t+1}}))_D
\]
can be chosen, where the indices are taken modulo $m$,
such that $e_1,\ldots,e_m$ are pairwise distinct. Suppose further that, for every $i\in[r]$,
there exist pairwise arc-disjoint $(x^i_t,y^i_t)$-trails
$T^i_t$ in $D_i$, over all indices $t$ with $i_t=i$,
such that $\bigcup_{t:i_t=i}V(T^i_t)=V(D_i)$. Then $D$ is supereulerian.
\end{lemma}

\begin{proof}
According to the assumption, $W=i_1i_2\cdots i_mi_1$
is a spanning closed trail on $[r]$, and for all indices t with \(i_t = i\), there exist pairwise arc‑disjoint \((x^i_t,y^i_t)\)‑trails \(T^i_t\) in \(D_i\) such that the union of \(\{T^i_t\}_t\) covers all vertices of \(D_i\). Hence, the concatenation
\[
T^{i_1}_1e_1T^{i_2}_2e_2\cdots T^{i_m}_me_m
\]
is a closed trail.
 Since $V(D)=V(D_1)\dot\cup\cdots\dot\cup V(D_r)$, this closed trail is spanning.
\end{proof}

 We use the following elementary observation from Yang, Liu and Meng \cite{YangLiuMeng2023}.

	\begin{lemma} \label{lemma:two-cliques}
		 \cite{YangLiuMeng2023} Let $G$ be a connected graph with $\kappa(G)<2$ and $\alpha(G)=2$. Then $V(G)$ can be partitioned into two nonempty sets, each of which induces a complete graph.
	\end{lemma}
	
\subsection{The construction of exceptional family \texorpdfstring{$\mathcal H$}{H}}

 We now define the exceptional family $\mathcal H$   appearing in Theorem~\ref{main}.

	\begin{definition}\label{def:H}
{ The family $\mathcal H$ consists of the digraphs $D$ satisfying $\lambda(D)\ge 2$ for which there is a partition
 $ V(D)=V(D_1)\dot\cup V(D_2)\dot\cup V(D_3)\dot\cup V(D_4)$
with the following properties.
\begin{enumerate}[label=\textnormal{(H\arabic*)}]
  \item Each $D_i=D[V(D_i)]$ is a nonempty complete  digraph.
  \item $|V(D_1)|=3$, and $V(D_1)=A\dot\cup B$ with $A,B\ne\varnothing$.
  \item For $i\in\{2,3,4\}$, $(A,V(D_i))_D=(V(D_i),B)_D=\varnothing,$   while
    $|(B,V(D_i))_D|\ge2$ and $ |(V(D_i),A)_D|\ge2.$
  \item For distinct $i,j\in\{2,3,4\}$, there are no arcs between $D_i$ and $D_j$ in either direction.
\end{enumerate}}
\end{definition}

	\begin{figure}[htbp]
		\centering
		\includegraphics[width=0.5\textwidth]{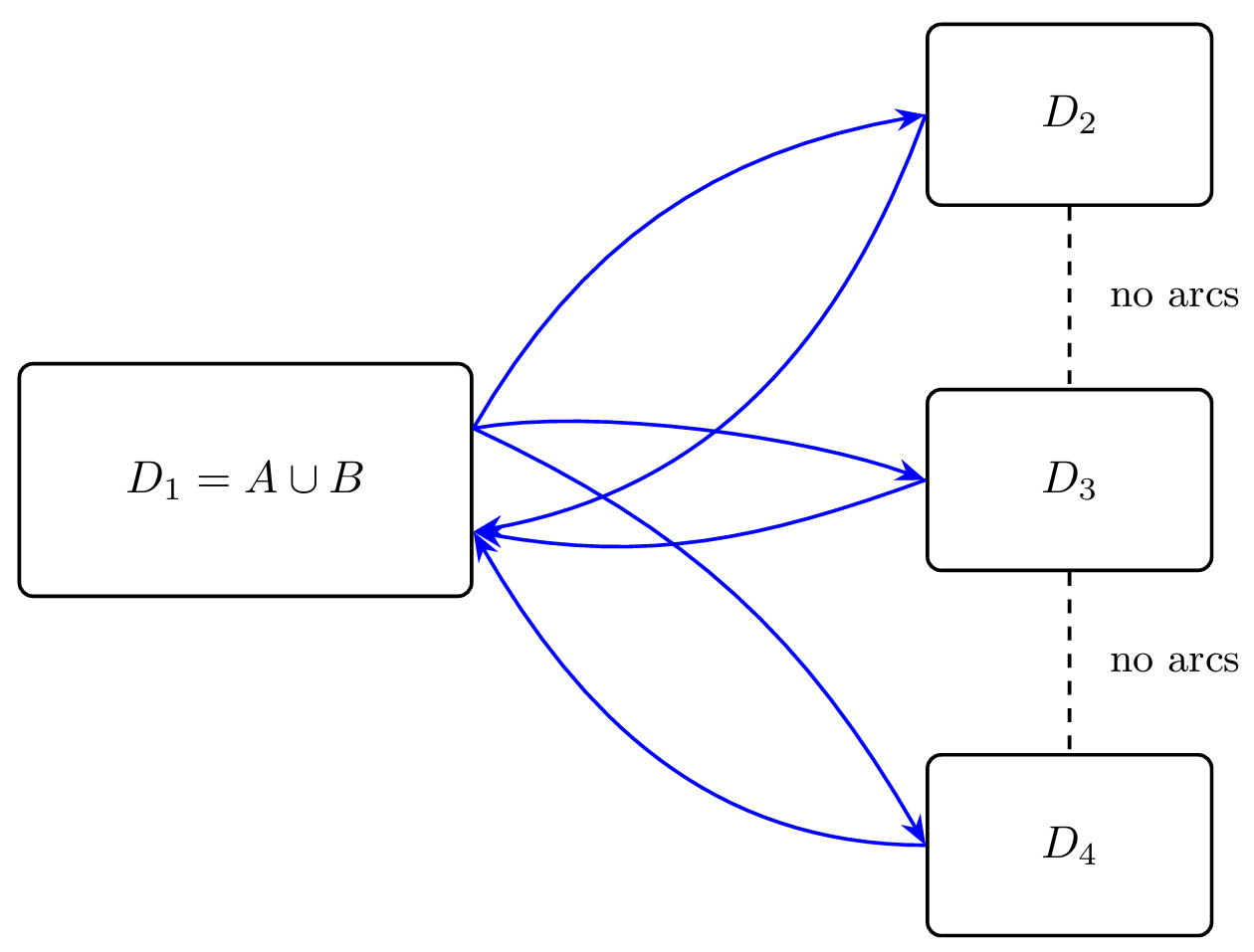}
        \caption{The exceptional family $\mathcal H$.}
    \end{figure}

	\begin{theorem} \label{theorem2.1}
		 Let $\mathcal{H}$ be the digraph family defined in Definition~\ref{def:H}. For every $D\in\mathcal H$, the following hold.
		\begin{itemize}
			\item[(i)] $\alpha_2(D)=4$.
			\item[(ii)] $\lambda(D)=2$.
			\item[(iii)] $D$ is not supereulerian.
		\end{itemize}
	\end{theorem}

	  \begin{proof}
	To prove (i), we can choose one vertex from each $D_i$ for $i \in \{1,2,3,4\}$. By (H3) and (H4), no two of these four vertices form a 2-cycle, so they form an $\alpha_2$-independent set of $D$.  Therefore $\alpha_2(D)\geq 4$. It is easy to check that there is no $\alpha_2$-independent set of five vertices by the construction of $D$, which implies that $\alpha_2(D)=4$.

For (ii), by Definition~\ref{def:H}, we have $\mid V(D_1)\mid =3$. Since $A\dot\cup B=V(D_1),$   we have $\mid (A,B)_D \mid=2 $. By (H3), there are no arcs from $A$ to
$V(D)\setminus V(D_1)$.
Since $D_1$ is complete and $
=|(A,B)_D|=2$, it follows that $\lambda(D)\leq2$.

Finally, we prove (iii). Note that any spanning closed trail must enter and leave each of $D_2$, $D_3$, $D_4$ at least once. By (H3) and (H4), every arc connecting $D_1$
with a leaf block must start in $B$ or end in $A$. Since there are no arcs between distinct leaf blocks, after leaving each leaf block, a spanning closed trail must traverse an arc from $A$ to $B$ before entering another leaf block. Thus, it requires at least three distinct arcs from $A$ to $B$. However, since $D_1$ is complete and
$|A||B|=2$, there are exactly two arcs
from $A$ to $B$, a contradiction.
\end{proof}

	\section{Proof of the key reduction lemma}

In this section, we prove Lemma~\ref{mainlemma}, which is used in the proof of Theorem~\ref{main}.

	\begin{lemma}
		Let $D$ be a strong digraph with $\lambda(D)\geq2$.
        \begin{itemize}

        \item[(1)] Suppose that $V(D)$ can be  partitioned by four vertex-disjoint complete subdigraphs $D_1$, $D_2$, $D_3$ and $D_4$. Then $D$ is supereulerian if and only if $D\notin\mathcal{H}$.

        \item[(2)] Suppose that $V(D)$ can be partitioned into
three vertex-disjoint symmetric subdigraphs
$D_1,D_2,D_3$, where, for some
$i\in\{1,2,3\}$, $\lambda(D_i)\geq2$,
and the other two subdigraphs are complete.
Then $D$ is supereulerian.

        \end{itemize}
		\label{mainlemma}
	\end{lemma}

\begin{proof}
		First we prove (1). Necessity is obvious, because if $D\in\mathcal H$, $D$ is not supereulerian by Theorem~\ref{theorem2.1} (iii). It remains to prove the sufficiency, that is, every $D\notin\mathcal H$ satisfying the hypotheses is supereulerian. We obtain the contracted digraph \(D'\) by contracting \(D_i\) to \(v_i\) for each \(i\in[4]\) and deleting parallel arcs. The contracted directed multigraph \(D''\) is defined by keeping all parallel arcs between distinct contracted vertices. Since $D$ is strong, both $D'$ and $D''$ are strong.
We use $D'$ to classify the possible configurations
and $D''$ to construct spanning closed trails. %We distinguish five cases according to the structure of the contracted digraph $D'$. Claim \ref{claim1} classifies the possible structures of \(D'\).

		\begin{claim}
			 If $D'$ contains no $4$-cycle, then, after relabelling, one of the following holds:
			\begin{itemize}			
				\item[(i)] $D'\cong H_1$, where $V(H_1)=V(D')$ and
				\[
				A(H_1)=\{(v_1,v_2),(v_2,v_1),(v_2,v_3),(v_3,v_2),(v_3,v_4),(v_4,v_3)\}.
				\]

				\item[(ii)] $D'\cong H_2$, where $V(H_2)=V(D')$,
					\[
					A(H_2)\supseteq \{(v_3,v_4),(v_4,v_3),(v_3,v_1),(v_1,v_2),(v_2,v_3)\},
					\]
					\[
					A(H_2)\cap \{(v_2,v_4),(v_4,v_1)\} =\varnothing,
					\]
					while each of $(v_1,v_3),(v_2,v_1),(v_3,v_2),(v_4,v_2),(v_1,v_4)$ may or may not be present.

				\item[(iii)] $D'\cong H_3$, where $V(H_3)=V(D')$ and
				\[
				A(H_3)\supseteq \{(v_1,v_2),(v_2,v_3),(v_3,v_4),(v_4,v_2),(v_3,v_1)\}.
				\]
                \[
				A(H_3)\cap \{(v_2,v_1),(v_1,v_3),(v_2,v_4),(v_4,v_3),(v_1,v_4),(v_4,v_1)\}
                =\varnothing,
				\]
                while $(v_3,v_2)$ may or may not be present.

				\item[(iv)] $D'\cong H_4$, where $V(H_4)=V(D')$ and
				\[
				A(H_4)=\{(v_1,v_2),(v_2,v_1),(v_2,v_3),(v_3,v_2),(v_2,v_4),(v_4,v_2)\}.
				\]

			\end{itemize}

			\label{claim1}
		\end{claim}

		\begin{figure}[htbp]
			\centering
			\begin{subfigure}[b]{0.22\textwidth}
				\centering
				\includegraphics[width=\linewidth]{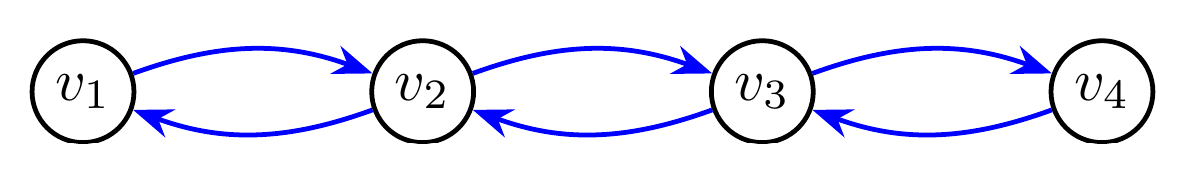}
		\caption{$H_1$ in class (i)}
			\end{subfigure}
			\hfill
			\begin{subfigure}[b]{0.22\textwidth}
				\centering
				\includegraphics[width=\linewidth]{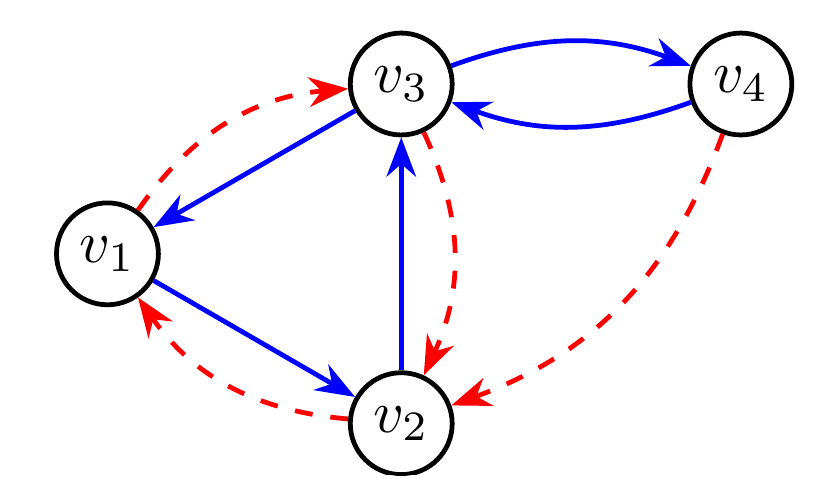}
				\caption{$H_2$ in class (ii)}
			\end{subfigure}
			\hfill
			\begin{subfigure}[b]{0.22\textwidth}
				\centering
				\includegraphics[width=\linewidth]{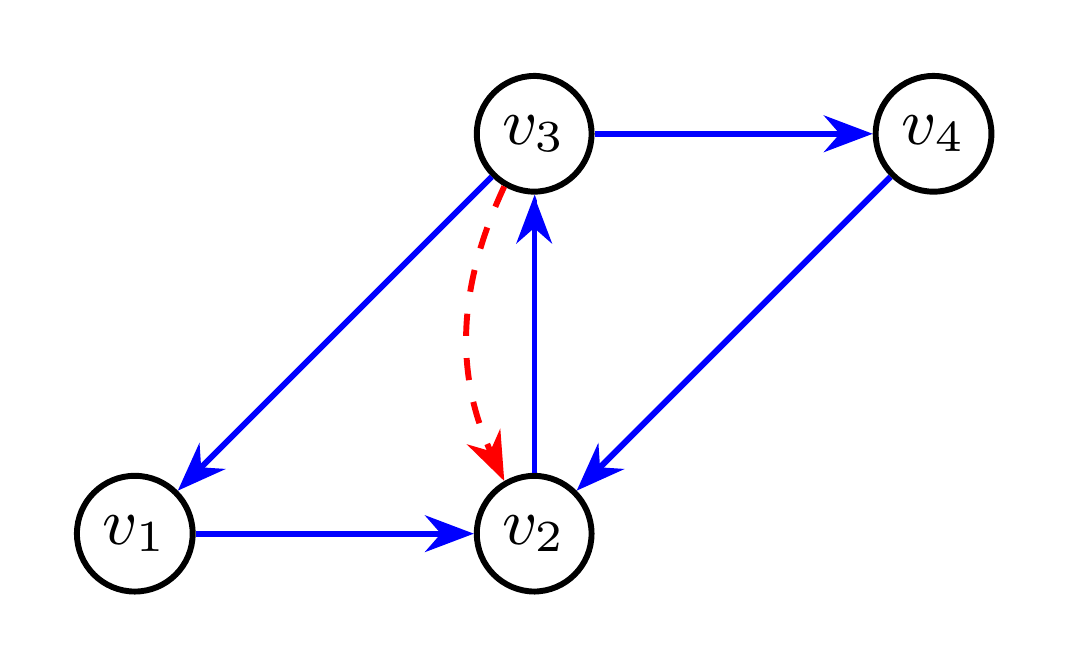}
				\caption{$H_3$ in class (iii)}
			\end{subfigure}
			\hfill
			\begin{subfigure}[b]{0.22\textwidth}
				\centering
				\includegraphics[width=\linewidth]{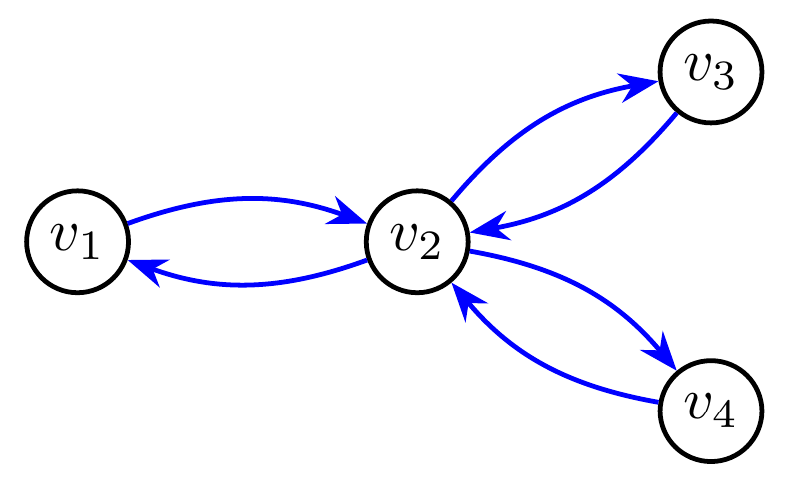}
				\caption{$H_4$ in class (iv)}
			\end{subfigure}
			\caption{The four configurations in Claim~\ref{claim1}. Blue arcs are required; red arcs are optional.}
\label{fig1}
		\end{figure}

\begin{proof}
			Since $D'$ is strong, it contains a 2-path, let $v_1v_2v_3$ be a 2-path. For the sake of contradiction, assume that $D'$ contains no $4$-cycle and does not satisfy any of the four configurations $H_1,H_2,H_3,H_4$. Additionally, we assume
\begin{equation}\label{(1)}
                    \text{$D'$ contains no 3-path.}
\end{equation}
Our aim is to show $D'\cong H_4$. By \eqref{(1)}, there is $(v_3,v_4),(v_4,v_1)\notin A(D')$. The connectivity of $D'$ implies that $D'$ contains a path $P$ from $v_3$ to $v_4$ and a path $Q$ from $v_4$ to $v_1$. The only possible
$(v_3,v_4)$-paths are
\[
v_3v_1v_4,\quad
v_3v_1v_2v_4,\quad
v_3v_2v_1v_4,\quad
v_3v_2v_4.
\]
If $P$ is of the form $v_3v_1v_4$ or $v_3v_1v_2v_4$ or $v_3v_2v_1v_4$, then $v_2v_3v_1v_4$ or $v_3v_1v_2v_4$ or $v_3v_2v_1v_4$ constitute a 3-path in $D'$, respectively, which gives a contradiction. Hence $P$ must be the path $v_3v_2v_4$. Further, \eqref{(1)} indicates again that $(v_1,v_3),(v_3,v_1)\notin A(D')$, since otherwise, $v_1v_3v_2v_4$ or $v_3v_1v_2v_4$ constitute a 3-path in $D'$, respectively, which gives a contradiction. Similarly, since $(v_4,v_1),(v_3,v_1)\notin A(D')$,
the only possible simple $v_4v_1$-paths are
$v_4v_2v_1$ and $v_4v_3v_2v_1$.
The latter is a 3-path, so
$Q=v_4v_2v_1$. Moreover, \eqref{(1)} implies that
$(v_4,v_3),(v_1,v_4)\notin A(D')$.
Indeed, otherwise $v_4v_3v_2v_1$ or
$v_1v_4v_2v_3$ is a 3-path in $D'$,
a contradiction. Hence, concatenating $v_1v_2v_3$ with $P$ and $Q$, we obtain $D'\cong H_4$, a contradiction. This implies that $D'$ contains a 3-path, say $v_1v_2v_3v_4$.

Since $D'$ has no $4$-cycle, $(v_4,v_1)\notin A(D')$. The connectivity of $D'$ shows that $v_4$ has an out-neighbour in $\{v_2,v_3\}$ and $v_1$ has an in-neighbour in $\{v_2,v_3\}$. Choose an arc from $v_4$ to $\{v_2,v_3\}$ and an arc from $\{v_2,v_3\}$ to $v_1$.
There are four possible pairs of these arcs. Firstly, suppose that $(v_4,v_2),(v_3,v_1)\in A(D')$. In this case, $(v_1,v_4)\notin A(D')$ otherwise $D'$ contains a directed 4-cycle $v_1v_4v_2v_3v_1$. Further, $(v_2,v_4)\notin A(D')$, as otherwise, we obtain $(v_4,v_3)\notin A(D')$ since $v_1v_2v_4v_3v_1$ is a directed $4$-cycle, which yields that $D'\cong H_2$ under the relabelling $(v_1,v_2,v_3,v_4):=(v_3,v_1,v_2,v_4)$, contradicting our assumption. Moreover, if $(v_4,v_3)\in A(D')$, then $D'\cong H_2$ under the labelling $(v_1,v_2,v_3,v_4):=(v_1,v_2,v_3,v_4)$. If $(v_1,v_3),(v_2,v_1)\in A(D')$, then $D'$ contains the directed $4$-cycle $v_1v_3v_4v_2v_1$. And if exactly one of $(v_1,v_3)$ and $(v_2,v_1)$ is present in $A(D')$, then $D'\cong H_2$ under the relabelling $(v_1,v_2,v_3,v_4):=(v_4,v_2,v_3,v_1)$ when $(v_1,v_3)\in A(D')$, or $(v_1,v_2,v_3,v_4):=(v_3,v_4,v_2,v_1)$ when $(v_2,v_1)\in A(D')$. Thus $(v_4,v_3)\notin A(D'),(v_1,v_3)\notin A(D')$, and $(v_2,v_1)\notin A(D')$. The remaining arcs give that $D'\cong H_3$, again a contradiction. Thus $(v_4,v_2)$ and $(v_3,v_1)$ cannot both be present in $D'$. Secondly, assume that $(v_3,v_1),(v_4,v_3)\in A(D')$. Then $(v_2,v_4)\notin A(D')$, since otherwise $D'$ contains a $4$-cycle. Hence $D'\cong H_2$, a contradiction. By the symmetric argument, if $(v_4,v_2),(v_2,v_1)\in A(D')$,  $D'\cong H_2$ under the relabelling $(v_1,v_2,v_3,v_4):=(v_3,v_4,v_2,v_1)$, again a contradiction. The only remaining possibility is $(v_4,v_2),(v_3,v_1)\notin A(D')$ and $(v_4,v_3),(v_2,v_1)\in A(D')$. The strong connectivity of $D'$ forces a $(v_4,v_1)$-path, and consequently $(v_3,v_2)\in A(D')$. Moreover, $(v_2,v_4),(v_1,v_3)\notin A(D')$ otherwise $D'\cong H_2$ under the relabelling $(v_1,v_2,v_3,v_4)
:=(v_4,v_3,v_2,v_1)$ if $(v_2,v_4)\in A(D')$ or $ (v_1,v_2,v_3,v_4)
:= (v_2,v_1,v_3,v_4)$ if $(v_1,v_3)\in A(D')$, a contradiction. This implies that $D'\cong H_1$, a contradiction. In summary, a contradiction arises in all four cases, which completes the proof.
		\end{proof}

We proceed by case analysis based on the structures of $D'$ given by Claim \ref{claim1} to show that $D$ admits a spanning closed trail.
 Suppose, to the contrary, that $D$ is not supereulerian.
By Claim~\ref{claim1}, either $D'$ contains a $4$-cycle or, after relabelling, $D'$ is one of $H_1,H_2$, $H_3$ and $H_4$. In each of these five possibilities, the directed multigraph $D''$ admits a spanning closed trail $W$, which we define as follows:
\[
W=
\begin{cases}
v_1v_2v_3v_4v_1, & \text{if } D' \text{ contains a 4-cycle, assume $v_1v_2v_3v_4v_1$},\\
v_1v_2v_3v_4v_3v_2v_1, & \text{if } D'\cong H_1,\\
v_3v_4v_3v_1v_2v_3, & \text{if } D'\cong H_2,\\
v_1v_2v_3v_4v_2v_3v_1, & \text{if } D'\cong H_3,\\
v_1v_2v_3v_2v_4v_2v_1, & \text{if } D'\cong H_4.
\end{cases}
\]
By the structure of $D'$ (see Fig. \ref{fig1}), the existence of $W$ is straightforward when $D'\not\cong H_3$. When $D'\cong H_3$, $W$ uses two distinct arcs from $v_2$ to $v_3$, and such arcs exist in $D''$ since $|(V(D_2),V(D_3))_D|\ge\lambda(D)\ge2$. Moreover, $W$ traverses each vertex $v_i$ with $i\in[4]$ at most three times.

%When $D'\cong H_3$, $\mid (V(D_2),V(D_3))_D\mid\geq 2$ since $\lambda(D)\geq2$. So it is easy to check that the directed multigraph $D''$ contains a spanning closed trail $W$ in either case, where $ W$ passes through each vertex $v_i$ at most 3 times, which corresponds to the spanning closed trail in Lemma \ref{lemma7}.

Let $\boldsymbol{\ell_i}$ denote the number of visits to the vertex $v_i$ in the selected closed trail $W$, i.e., $0\leq \ell_i\leq 3$. Recall that $W$ is a spanning closed trail of $D''$, for convenience, assume that $W:=v_{j_1}v_{j_2}\cdots v_{j_m}v_{j_1} $, where $j_i\in [4]$ for each $i\in [m]$. For each $t\in[m]$, let $e_t$ be the arc $ v_{j_t}v_{j_{t+1}}$ corresponding to the $t$-th arc of $W$ in $D''$, let $x^{j_{t+1}}_t$ be the corresponding vertex of terminal vertex of $e_t$ in $D_{j_{t+1}}$ and let $y^{j_{t+1}}_t$ be the corresponding vertex of initial vertex
 of $e_{t+1}$ in $D_{j_{t+1}}$. The pair $(x^{j_{t+1}}_t,y^{j_{t+1}}_t)$ is associated with
the visit to $D_{j_{t+1}}$ immediately following the
arc $e_t$. Thus, after a cyclic shift of the
indices of $W$, the resulting pairs satisfy
the hypotheses of Lemma~\ref{lemma7}. Since $W$ is a trail in $D''$,
the arcs $e_1,\ldots,e_m$ are pairwise distinct. To apply Lemma \ref{lemma7} to obtain the desired spanning closed trail of $D$, we next need to show, in each $D_i$, there exist pairwise arc‑disjoint \((x^i_t,y^i_t)\)‑trails \(T^i_t\) in \(D_i\) such that the union of \(\{T^i_t:t\in [m]\}\) covers all vertices of \(D_i\).

If $|V(D_i)|\geq \ell_i+1$ or $|V(D_i)|=1$,
then Lemma~\ref{lemma4}, applied with
$s=\ell_i$, gives the required pairwise
arc-disjoint trails covering $V(D_i)$. Consequently, if $D_i$ satisfies $|D_i|\geq \ell_i+1$ or $|D_i|=1$ or $\ell_i=1$, for each $i\in [4]$, we may apply Lemma \ref{lemma4} and then Lemma \ref{lemma7} to obtain the a spanning closed trail of $D$, a contradiction. Consequently, we may choose an index $i\in[4]$
with $1<|V(D_i)|\leq\ell_i$ such that no
$\ell_i$ pairwise arc-disjoint
$(x^i_t,y^i_t)$-trails cover $V(D_i)$.
We arbitrarily fix such an index $i\in [4]$ throughout the following
argument. Further, assume that
\begin{equation}\label{cons}
  \begin{aligned}
  &\text{there is no $\ell_i$ pairwise arc-disjoint $(x^i_j,y^i_j)$-trails $T^i_{(x^i_j,y^i_j)}$, $j\in[\ell_i]$,}\\
   &\text{\quad \quad \quad \ \ \ \ \ \  such that $\bigcup_{j=1}^{\ell_i} V\bigl(T^i_{(x^i_j,y^i_j)}\bigr)=V(D_i)$.}
  \end{aligned}
\end{equation}

% Consider each $i\in [4]$ with $|D_i|=2$.  Namely, the following claim holds.

\begin{claim}\label{obver}
Suppose that $|V(D_i)|=2$ and
$2\leq\ell_i\leq3$, and let
$V(D_i)=\{a_i,b_i\}$.
Assume that no $\ell_i$ pairwise arc-disjoint
$(x^i_t,y^i_t)$-trails cover $V(D_i)$.
Then there exist two distinct indices $t,q$
with $j_{t+1}=j_{q+1}=i$ such that, after
interchanging $a_i$ and $b_i$ if necessary,
 $(x^i_{t},y^i_{t}) = (x^i_{q},y^i_{q})=(a_i,b_i)$ (i.e., $ j_{t+1}=i$). Moreover $|(a_i,V(D_{j_{t+2}})\cup V(D_{j_{q+2}})_D|=0$  and $|(V(D_{j_{t}})\cup V(D_{j_{q}}),b_i)_D|=0$.
 \end{claim}

\begin{proof}
By Lemma~\ref{lemma5}, since there is no $\ell_i$ pairwise arc-disjoint
$(x^i_t,y^i_t)$-trails cover $V(D_i)$, at least one of the
pairs $(a_i,b_i)$ and $(b_i,a_i)$ occurs at least
twice. After interchanging $a_i$ and $b_i$ if
necessary, choose distinct indices $t,q$ such that $(x^i_t,y^i_t)=(x^i_q,y^i_q)=(a_i,b_i).$

We first consider the case in which
$\ell_i=3$ and all three terminal pairs are
$(a_i,b_i)$. This can occur only when
$D'\cong H_4$. Since $\lambda(D)\geq2$,
there exist external arcs $a_i z$ and $wb_i$.
Let $D_{\mathrm{out}}$ and $D_{\mathrm{in}}$
be the outer parts containing $z$ and $w$,
respectively. Choose one of the two cyclic
orders of the three outer parts so that the
two arcs correspond to distinct visits to
$D_i$. Replace the corresponding two connecting
arcs in $W$ by $a_i z$ and $wb_i$.
The three terminal pairs in $D_i$ then become $(a_i,a_i), (b_i,b_i), (a_i,b_i).$ By Lemma~\ref{lemma5}, the required three
pairwise arc-disjoint trails covering $D_i$
exist. Each of the other three parts is visited
once, so Lemma~\ref{lemma4} provides its
required local trail. Lemma~\ref{lemma7}
therefore gives a spanning closed trail of $D$,
a contradiction.

We may thus assume that the third terminal pair,
if $\ell_i=3$, is different from $(a_i,b_i)$. Suppose first that there is an arc $a_i v\in A(D)$ with $v\in V(D_{j_{s+2}})$, for some $s\in\{t,q\}$. Replace the connecting
arc $e_{s+1}$ in $W$ by $a_i v$. Thus $y^i_s:=a_i$ and $x^{j_{s+2}}_{s+1}:=v. $ The terminal pair $(a_i,b_i)$ at this visit
becomes $(a_i,a_i)$. Hence the terminal pairs
in $D_i$ satisfy Lemma~\ref{lemma5}, and the
required local trails covering $D_i$ exist.

We show that the other parts also admit the
required local trails. If $D'\cong H_2$ or $D'\cong H_4$, every part
other than $D_i$ is visited once. Thus
Lemma~\ref{lemma4} applies. Suppose that $D'\cong H_1$, and let $Q$ be the unique member of
$\{D_2,D_3\}\setminus\{D_i\}$. If $Q$ admits the required local trails after the preceding
replacement, there is nothing to prove.
Otherwise, Lemmas~\ref{lemma4} and
\ref{lemma5} imply that $|V(Q)|=2$ and, after
relabeling $V(Q)=\{c,d\}$, its two terminal
pairs are $(c,d),(c,d).$ Since $d_D^+(c)\geq2$, there is an external
arc $cw\in A(D)$ with $w\neq d$. By the structure of $H_1$,
the vertex $w$ belongs either to $D_i$ or to
the remaining outer part, which is visited
once by $W$. Replace the corresponding
connecting arc leaving $Q$ by $cw$.
The two terminal pairs in $Q$ become $(c,c),(c,d),$
so Lemma~\ref{lemma5} applies to $Q$. Notice that if $w\in D_i$, this replacement changes the original terminal pairs $(a_i,a_i), (a_i,b_i)$ in $D_i$ to either $\{(u,a_i),(a_i,b_i)\}$ or $\{(a_i,a_i),(u,b_i)\}$ for some $u\in\{a_i,b_i\}$. Lemma~\ref{lemma5} applies in both cases. If $w$ belongs to the remaining outer part, Lemma~\ref{lemma4} applies to that part.

Suppose next that $D'\cong H_3$.
If $i=3$, the two successor parts of $D_i$
are $D_1$ and $D_4$, both of which are visited
once, and Lemma~\ref{lemma4} applies.
If $i=2$, the affected part is $D_3$.
When $|V(D_3)|\geq3$, Lemma~\ref{lemma4}
applies. Otherwise, let $V(D_3)=\{c,d\}$.
The two arcs of $W$ from $D_2$ to $D_3$
are distinct and both have initial vertex
$b_i$. Since $D$ has no multiple arcs, their
terminal vertices are distinct, and hence
are precisely $c,d$. We replace the arc of $W$ whose
terminal vertex is $v$ by $a_i v$.
Consequently, two
terminal pairs in $D_3$  are unchanged and
have distinct first coordinates. Lemma~\ref{lemma5}
gives the required local trails in $D_3$. Thus, in every case, the resulting
closed trail $W'$ is still a spanning closed
trail in $D''$. By Lemma~\ref{lemma7}, $D$
has a spanning closed trail, a contradiction. Therefore,
$(a_i,V(D_{j_{t+2}})\cup V(D_{j_{q+2}}))_D
=\varnothing.$ Symmetrically,
 $(V(D_{j_t})\cup V(D_{j_q}),b_i)_D
=\varnothing.$ This completes the proof.
\end{proof}

Indeed, $D'$ is directed 4-cycle-free, as otherwise $\ell_i=1$ for every $i\in[4]$, which contradicts $1<|D_i|\leq \ell_i$ for some $i\in [4]$. The rest of the proof proceeds by considering the four possible structures of $D'$ by Claim \ref{claim1}.

\medskip
\noindent\textbf{Case 1.} Suppose $D'\cong H_1$. The construction of $W$ implies that $\ell_2=\ell_3=2$ and $\ell_1=\ell_4=1$. If $|V(D_2)|\neq 2$ and  $|V(D_3)|\neq 2$, Lemmas~\ref{lemma4} and~\ref{lemma7} give a spanning closed trail of $D$, a contradiction. Hence $|V(D_2)|=2$ or $|V(D_3)|=2$. For each index $j\in\{2,3\}$ with $|V(D_j)|=2$, let $V(D_j)=\{a_j,b_j\}$. Notice that $v_{j-1}v_jv_{j+1}$ and $v_{j+1}v_jv_{j-1}$ are the two subtrails of $W$ visiting $D_j$. Thus $|(a_j,V(D_{j+1})\cup V(D_{j-1})_D|=0$ and $|(V(D_{j+1})\cup V(D_{j-1}),b_j)_D|=0$. Together with the construction of $H_1$, we obtain that $d^-_D(b_j)=d^+_D(a_j)=1$, which contradicts $\lambda(D)\geq 2$.

\medskip
\noindent\textbf{Case 2.} Suppose $D'\cong H_2$. The construction of $W$ implies that $\ell_3=2$. If $|V(D_3)|\ne2$, Lemmas~\ref{lemma4} and~\ref{lemma7}
give a spanning closed trail of $D$, a contradiction. Thus let $V(D_3)=\{a,b\}$. In this case, $e_2e_3(=v_4v_3v_1)$
and $e_5e_1(=v_2v_3v_4)$
are the two subtrails of $W$ visiting $D_3$.  Therefore, $(x^3_{2},y^3_{2}) = (x^3_{5},y^3_{5})=(a,b)$, $|(V(D_2)\cup V(D_4),b)_D|=0$ and $|(a,V(D_4)\cup V(D_1))_D|=0$ by Claim \ref{obver}. Set $S=V(D_2)\cup\{a\}$. By the arc set of $H_2$,
\[
2\le |(S,V(D)\setminus S)_D| =1+|(V(D_2),V(D_1))_D|.
\]
Thus $(V(D_2),V(D_1))_D\ne\varnothing$. Since $\lambda(D)\geq 2$ and the structure of $H_2$, $|(V(D_1),b)_D|\geq 1, |(a,V(D_2))_D|\geq 1$. This implies that there exist vertices with $\{u,u'\}\subseteq V(D_2)$ and $\{v,v'\}\subseteq V(D_1)$ such that $(u,v)\in A(D)$, $(v',b)\in A(D)$ and $(a,u')\in A(D)$. Let $e'_1:=(v_2,v_1)$, $e'_2:=(v_1,v_3)$ and $e'_5:=(v_3,v_2)$ denote the arcs of $D''$ corresponding to the three arcs in $D$, respectively. Denote $W':=e'_1e'_2e_1e_2e'_5=v_2v_1v_3v_4v_3v_2$.  The two subtrails of $W'$ visiting $D_3$ are $e'_2e_1$ and $e_2e'_5$. The terminal vertex of $e'_2$ and the initial vertex of $e_1$ in $D_3$ both correspond to $b\in V(D)$; the terminal vertex of $e_2$ and the initial vertex of $e_5'$ in $D_3$ both correspond to $a\in V(D)$. Consequently, $(a,a)$ and $(b,b)$ form two arc-disjoint trails covering $V(D_3)$. Apply Lemma \ref{lemma7} with $W'$, there exists a spanning closed trail of $D$, a contradiction.

\medskip
\noindent\textbf{Case 3.} Suppose $D'\cong H_3$. The construction of $W$ implies that  $\ell_2=\ell_3=2$. Analogously, if $|V(D_2)|\neq 2$ and  $|V(D_3)|\neq 2$, Lemmas~\ref{lemma4} and~\ref{lemma7} give a spanning closed trail of $D$, a contradiction. Hence $|V(D_2)|=2$ or $|V(D_3)|=2$. Then the index $i\in\{2,3\}$. If $i=2$, let $V(D_2)=\{a,b\}$. Notice that $v_1v_2v_3$ and $v_4v_2v_3$  are the two subtrails of $W$ visiting $D_2$. Thus $|(a,V(D_{3}))_D|=0$ by Claim \ref{obver}. Together with the construction of $H_3$, we obtain that $d^+_D(a)=1$, which contradicts $\lambda(D)\geq 2$. Thus, $i=3$. Assume that $D_3=\{a,b\}$, the predecessor part for both visits is $D_2$. Claim \ref{obver} gives $(D_2,b)_D=\varnothing$. Moreover, $H_3$ forbids $D_1\to D_3$, $D_4\to D_3$. Hence $d_D^-(b)=1$, contradicting $\lambda(D)\ge 2$.

\medskip
\noindent\textbf{Case 4.} Suppose $D'\cong H_4$. The construction of $W$ implies that $\ell_2=3$. If $|V(D_2)|\ge 4$, Lemmas~\ref{lemma4} and~\ref{lemma7}
give a spanning closed trail of $D$, a contradiction. Thus $|V(D_2)|=2$ or $|V(D_2)|=3$.  In this case, $e_1e_2(=v_1v_2v_3)$, $e_3e_4(=v_3v_2v_4)$, and $e_5e_6(=v_4v_2v_1)$ are the three subtrails of $W$ visiting $D_2$. Our aim is to find three arc disjoint $(x^2_1,y^2_1)$-trail $T^2_{1}$, $(x^2_3,y^2_3)$-trail $T^2_{3}$, and $(x^2_5,y^2_5)$-trail $T^2_{5}$, such that $ V(T^2_{1})\cup V(T^2_{3})\cup V(T^2_{5})=V(D_2)$. %这个目标可以达到： by Lemma~\ref{lemma5} if $|D_2|=2$ and $(x^2_1,y^2_1)$,  $(x^2_3,y^2_3)$ and $(x^2_5,y^2_5)$ 是两两不相同的点对儿；by Lemma~\ref{lemma6} 因此假设至少存在两对是完全相同的，不失一般性，假设 $(x^2_1,y^2_1)=(x^2_3,y^2_3)=(a,b)$ %Therefore, $|(V(D_2)\cup V(D_4),b)_D|=0$ and $|(a,V(D_4)\cup V(D_1))_D|=0$ by Claim \ref{cons}. Set $S=V(D_2)\cup\{a\}$. By the arc set of $H_2$,
% Set $\mathcal L=\{D_1,D_3,D_4\}$. For each $L\in\mathcal L$, define
% \[
% X_L=\{v\in V(D_2):(V(L),v)_D\ne\varnothing\},\qquad
% Y_L=\{v\in V(D_2):(v,V(L))_D\ne\varnothing\}.
% \]
We first assume $|V(D_2)|=2$, and let $V(D_2)=\{a,b\}$. According to Claim \ref{obver},  there exist two subtrails  of $W$ visiting $D_i$, say $e_1e_2(=v_1v_2v_3)$ and $e_3e_4(=v_3v_2v_4)$, with $(x^2_{1},y^2_{1}) = (x^2_{3},y^2_{3})=(a,b)$, $|(V(D_1)\cup V(D_3),b)_D|=0$ and $|(a,V(D_3)\cup V(D_4))_D|=0$. Consequently, there is no directed path from $a$ to $b$ in $D-(a,b)$, which contradicts the hypothesis $\lambda(D)\geq 2$. Hence, $|V(D_2)|\neq 2$, yielding $|V(D_2)|=3$.

Let $\mathcal L=\{D_1,D_3,D_4\}$. Define $X^*= \{z\in V(D_2): \exists L\in\mathcal L,\ (V(L),z)_D\ne\varnothing\}$, and $Y^*= \{z\in V(D_2): \exists L\in\mathcal L,\ (z,V(L))_D\ne\varnothing\}$. Here $X^*$ collects all terminals of external arcs entering $D_2$, and $Y^*$ collects all origins of external arcs leaving $D_2$. We then prove $X^*\cap Y^*=\varnothing$. Suppose there exists $z\in X^*\cap Y^*$. Then there are two arcs $uz$, $zv$, where $u$ and $v$ belong to some two distinct outer parts. We replace the corresponding incoming and outgoing arcs in the spanning closed trail of $H_4$. This yields a new spanning closed trail $W'$, whose three terminal pairs in $D_2$ satisfy $\{x_1,x_2,x_3\}\cap \{y_1,y_2,y_3\}\ne\varnothing$. By Lemma~\ref{lemma6}, there exist three arc-disjoint specified-terminal trails covering $D_2$. The other three parts are visited exactly once, and Lemma~\ref{lemma4} provides their respective spanning local trails. An application of Lemma~\ref{lemma7} gives a spanning closed trail of $D$, which is a contradiction. Hence $X^*\cap Y^*=\varnothing$.

We next construct a partition satisfying the definition of $\mathcal H$. Let $A=X^*$, $B=V(D_2)\setminus X^*$. By definition, $(V(L),B)_D=\varnothing$ for each $L\in\mathcal L$. Since $X^*\cap Y^*=\varnothing$, we have $(A,V(L))_D=\varnothing$ for each $L\in\mathcal L$. Moreover, $A$ and $B$ are both nonempty, because $W$ contains arcs entering and leaving $D_2$. We finally verify the two lower bounds on arc numbers. Since there are no arcs among the three outer parts in $H_4$, together with $\lambda(D)\ge2$, for each $L\in\mathcal L$, we have $|(B,V(L))_D| =|(V(D)\setminus V(L),V(L))_D| \ge2$, and $|(V(L),A)_D| =|(V(L),V(D)\setminus V(L))_D| \ge2$. Thus condition (H3) holds. The remaining conditions are checked as follows. (H1): $D_2$ is a complete digraph of order three. (H2): $A,B$ are nonempty and partition $V(D_2)$. (H4): there are no arcs among the three outer parts. Therefore $D\in\mathcal H$, contradicting $D\notin\mathcal H$.

%By Lemma~\ref{lemma6} and \eqref{cons}, we further derive
% \[
% \{x^2_{1},x^2_{3},x^2_{5}\}\cap\{y^2_{1},y^2_{3},y^2_{5}\}=\emptyset.
% \]

% Otherwise, there exists some vertex $x^2_i=y^2_j$, which prevents the existence of a nonempty proper subset $A\subset V(K)$ satisfying the conditions of Lemma~\ref{lemma6}. Let $A=\{x^2_{1},x^2_{3},x^2_{5}\}$ and let $B:=V(D_2)\setminus A$ implies $D\in\mathcal H$ by Definition~\ref{def:H}, since (H1),(H2),(H4) are obvious and (H3) follows by $\lambda(D)\geq2$, which gives the desired contradiction.

All four cases lead to contradictions. Therefore, every digraph $D\notin\mathcal H$ that satisfies the given hypotheses is supereulerian, which completes the proof of (1).

To prove (2), similarly, construct the contracted digraph $D'$ from $D$ by contracting each \(D_i\) to a single vertex \(v_i\) for \(i\in[3]\) without parallel arcs. Since $D'$ is a strong digraph on three vertices,
\begin{equation}\label{3}
{\text{either $D'$ contains a 3-cycle}}
\end{equation}
or, after relabelling, $D'$ is the bidirected path with arc set
\begin{equation}\label{4}
{\text{$A(D')=\{(v_1,v_2),(v_2,v_1),(v_2,v_3),(v_3,v_2)\}$}}
\end{equation}

 It is easy to check that $D$ is supereulerian whenever $D'$ contains a $3$-cycle by Lemma \ref{lemma1} and Lemma \ref{lemma7}. We thus suppose that $D'$ satisfies (\ref{4}). Consider the spanning closed trail  $W=v_1v_2v_3v_2v_1$ in $D'$, which yields $\ell_2=2$ and $\ell_1=\ell_3=1$. If $D_2$ is a symmetric digraph with $\lambda(D_2)\geq 2$, there exists a spanning closed trail of $D$ by applying Lemma~\ref{lemma1} with $s:=2,D_i:=D_2$, Lemma~\ref{lemma1} with $s:=1,D_i:=D_1$, Lemma~\ref{lemma1} with $s:=1,D_i:=D_3$ and then Lemma \ref{lemma7},  a contradiction. If $D_2$ is complete with $|D_2|\neq 2$, there exists a spanning closed trail of $D$ by applying Lemma~\ref{lemma4} with $s:=2,D_i:=D_2$, Lemma~\ref{lemma1} with $s:=1,D_i:=D_1$, Lemma~\ref{lemma1} with $s:=1,D_i:=D_3$ and then Lemma \ref{lemma7},  a contradiction.  Only the case that \(D_2\) is a complete digraph of order 2 is considered. Similarly as Claim \ref{obver}, both subtrails $v_1v_2v_3$ and $v_3v_2v_1$ of $W$ uses identical vertex pair of $D_2=\{a,b\}$, since otherwise, there exists a spanning closed trail of $D$ by applying Lemma~\ref{lemma5} with $s:=2,D_i:=D_2$, Lemma~\ref{lemma1} with $s:=1,D_i:=D_1$, Lemma~\ref{lemma1} with $s:=1,D_i:=D_3$ and then Lemma \ref{lemma7},  a contradiction. Also  $|(a,V(D_{1})\cup V(D_{3}))_D|=0$ and $|(V(D_{1})\cup V(D_{3}),b)_D|=0$. Consequently, $d_D^-(b)=d_D^+(a)=1$, contrary to $\lambda(D)\ge2$.
\end{proof}

	\section{Proof  of Theorem \ref{main}}

	 \begin{proof}[Proof of Theorem 1.2]
		Since each digraph $D\in \mathcal{H}$ satisfies $\lambda(D)=2$, it suffices to prove part (i). Necessity is obvious, since $D\notin\mathcal H$
whenever $D$ is supereulerian by
Theorem~\ref{theorem2.1}(iii). It remains to prove sufficiency. Suppose that $D\notin\mathcal H$. Consider the  underlying graph $G_D$ of the symmetric core $J(D)$. Since $\alpha(G_D)=\alpha_2(D)=4$, the graph $G_D$ has at most four connected components.  We distinguish cases according to the number of connected components.

		 If  $G_D$  is connected,  then  $J(D)$ is a strong spanning symmetric subdigraph of $D$ and hence is supereulerian. Therefore $D$ is supereulerian.

    Suppose  that $G_D$ has two connected components $G_1$ and $G_2$.  Then $J_1:=J(D)[V(G_1)]$ and $J_2:=J(D)[V(G_2)]$ are vertex-disjoint strong symmetric subdigraphs whose vertex sets partition $V(D)$. Since $D$ is strong,  there are arcs $(x_1,x_2)$ and $(x_3,x_4)$ with $x_1,x_4\in V(G_1)$ and $x_2,x_3\in V(G_2)$. By Lemma~\ref{lemma1} (with singleton components serving as trivial trails), $J_1$ has a spanning $(x_4,x_1)$-trail $T^1$, and $J_2$ has a spanning $(x_2,x_3)$-trail $T^2$, which are respectively spanning trails of $D[V(G_1)]$ and $D[V(G_2)]$. Hence $(x_1,x_2)T^2(x_3,x_4)T^1$ is a spanning closed trail of $D$, and $D$ is supereulerian.

	Suppose that $G_D$  has three connected components  $G_1,G_2,G_3$. Since $\alpha(G_D)=4$, we may assume that $\alpha(G_1)=\alpha(G_2)=1$ and $\alpha(G_3)=2$.  Thus $G_1$ and $G_2$ are complete. If $G_3$ has a bridge, then $\kappa(G_3)<2$, and Lemma~\ref{lemma:two-cliques} partitions $V(G_3)$ into two cliques. We then obtain four complete symmetric subdigraphs, so Lemma~\ref{mainlemma}(1) and the assumption $D\notin\mathcal H$ imply that  $D$ is supereulerian. If $G_3$ has no bridge, then the symmetric digraph $J_3:=J(D)[V(G_3)]$ has $\lambda (J_3)\geq 2$, and Lemma~\ref{mainlemma}(2) yields the same conclusion.

	Finally, suppose that $G_D$  has four connected components.  Since $\alpha(G_D)=4$, each component  has independence number one and hence  is complete.  Lemma~\ref{mainlemma}(1), together with $D\notin\mathcal H$, therefore implies that $D$ is supereulerian. Thus  $D$ is supereulerian  in every case.
	\end{proof}

	\section{Open  problems}
	In this paper, we characterize supereulerianity for 2-arc-strong digraphs $D$ with $\alpha_2(D)=4$. Very recently, Liu, Zhao, Yang and Lai \cite{LiuZhaoYangLai2027} gave a characterization for the case $\alpha_2(D)=3$. It is therefore natural to consider the general characterization of supereulerianity for strong digraphs $D$ with $\alpha_2(D)=a$.

	\begin{problem}
		For each integer $a\geq5$, characterize all
nonsupereulerian digraphs $D$ with
$\alpha_2(D)=a$ and $\lambda(D)\geq a-2$. Furthermore, is every strong digraph $D$
with $\alpha_2(D)=a$ and
$\lambda(D)\geq a-1$ supereulerian?
	\end{problem}

	A substantially more difficult direction is to replace $\alpha_2(D)$ by the independence number $\alpha_0(D)$.

	\begin{problem}
		 For each integer $a\ge2$, determine whether there exists a function $g(a)$ such that every strong digraph  $D$  with  $\alpha_0(D)\le a$ and $\lambda(D)\ge g(a)$ is supereulerian.  If such a function exists, determine or bound the smallest possible $g(a)$.
 \end{problem}

	 Even the case $\alpha_0(D)=2$ remains open in closely related forms of the Bang-Jensen--Thomass\'e conjecture; see \cite{BangJensenMaddaloni2015,DongLiuMeng2020}.
\section*{Acknowledgments}	

{The author’s work is supported by National Natural Science Foundation of China (No. 12571373).}

\end{document}